\documentclass{amsart}
\usepackage{epsfig}

\newtheorem{theorem}{Theorem}[section]
\newtheorem{lemma}[theorem]{Lemma}
\newtheorem{corollary}[theorem]{Corollary}

\theoremstyle{definition}

\theoremstyle{remark}

\numberwithin{equation}{section}

\begin{document}

\title[On a result of Koecher]
{On a result of Koecher concerning Markov-Ap\'ery type formulas for the Riemann
zeta function}

\author{Karl Dilcher}
\address{Department of Mathematics and Statistics\\
 Dalhousie University\\
         Halifax, Nova Scotia, B3H 4R2, Canada}
\email{dilcher@mathstat.dal.ca}

\author{Christophe Vignat}
\address{LSS-Supelec, Universit\'e Paris-Sud, Orsay, France and Department of
Mathematics, Tulane University, New Orleans, LA 70118, USA}
\email{cvignat@tulane.edu}
\keywords{Riemann zeta function, Euler sum, Markov-Ap\'ery identity}
\subjclass[2010]{Primary: 11M06; Secondary: 11M32, 33C05}
\thanks{The first author was supported in part by the Natural Sciences and 
Engineering Research Council of Canada}

%\date{November 19, 2019}

\setcounter{equation}{0}

\begin{abstract}
Koecher in 1980 derived a method for obtaining identities for the Riemann zeta
function at odd positive integers, including a classical result for $\zeta(3)$
due to Markov and rediscovered by Ap\'ery. In this paper we extend Koecher's 
method to a very general setting and prove two more specific but still rather
general results. As applications we obtain infinite classes of identities for
alternating Euler sums, further Markov-Ap\'ery type identities, and identities
for even powers of $\pi$.
\end{abstract}

\maketitle

\section{Introduction}

At least since Ap\'ery's proof of the irrationality of $\zeta(3)$ there has 
been a great deal of interest in identities such as 
\begin{equation}\label{1.1}
\zeta(3) = \frac{5}{2}\sum_{k=1}^\infty\frac{(-1)^{k-1}}{\binom{2k}{k}k^3}.
\end{equation}
This identity was actually first obtained by Markov in 1890; see, e.g., 
\cite{KS} which provides a historical perspective. Following Ap\'ery's work,
numerous other related identities were obtained or rediscovered, some of them
of a more general nature. For instance, Koecher \cite{Ko} used the generating
function 
\begin{equation}\label{1.2}
\sum_{k=0}^\infty\zeta(2k+3)x^{2k} = \sum_{n=1}^\infty\frac{1}{n(n^2-x^2)}
\qquad (|x|<1)
\end{equation}
as motivation to prove the identity 
\begin{equation}\label{1.3}
\sum_{n=1}^\infty\frac{1}{n(n^2-x^2)} 
= \frac{1}{2}\sum_{k=1}^\infty\frac{(-1)^{k-1}}{\binom{2k}{k}k^3}\cdot
\frac{5k^2-x^2}{k^2-x^2}\prod_{m=1}^{k-1}\big(1-\frac{x^2}{m^2}\big).
\end{equation}
The generating function \eqref{1.2} is easy to obtain by expanding each summand
on the right as a geometric series and changing the order of summation. 
Koecher obtained \eqref{1.1} by expanding the right-hand side of \eqref{1.3} 
in powers of $x$ and equating the constant coefficients in \eqref{1.2} and
\eqref{1.3}. Similarly, by equating coefficients of $x^2$, he obtained
\begin{equation}\label{1.4}
\zeta(5) = 2\sum_{k=1}^\infty\frac{(-1)^{k-1}}{\binom{2k}{k}k^5}
-\frac{5}{2}\sum_{k=1}^\infty\frac{(-1)^{k-1}H_{k-1}^{(2)}}{\binom{2k}{k}k^3},
\end{equation}
where $H_n^{(2)}=\sum_{j=1}^nj^{-2}$ is a generalized harmonic number.

Koecher also gave a corresponding formula for $\zeta(7)$, and it is clear that
such identities can be obtained from \eqref{1.2} and \eqref{1.3} for any
further odd argument of the zeta function. It should also be noted that
independently, and using a different method, Leshchiner \cite{Le} obtained the
same expanded identities as did Koecher.

The current paper is based on the observation that Koecher's proof is 
``structural", that is, its main ideas can be generalized. It is the purpose
of this paper to prove such a generalization, and to consider some special
cases. These, in turn, will lead to various identities or classes of identities,
some of which are new.

This paper is structured as follows. After we prove our main result in 
Section~2, we derive two different instances in Section~3, one of which includes
Koecher's original result \eqref{1.3} as a special case. The remainder of the
paper is devoted to applications of these results: In Section~4 we 
obtain a general result on alternating Euler sums, and in Section~5 we derive
an infinite class of Markov-Ap\'ery type formulas which generalize the 
identity \eqref{1.1}. Finally, in Section~6, we derive a class of identities
for positive even powers of $\pi$.

\section{The main result}

The basic idea in our approach is to replace the sequence of positive integer
squares, which plays an important role in Koecher's proof, by an arbitrary
increasing sequence
\begin{equation}\label{2.1}
{\bf z} := (z_1, z_2, z_3,\ldots),\quad 0 < z_1 < z_2 <\cdots
\end{equation}
of positive real numbers, with the additional condition
\begin{equation}\label{2.2}
z_n \geq \varepsilon\,n, \quad n\geq 1,\quad\hbox{for some}\quad \varepsilon >0.
\end{equation}
Extending Koecher's notation in \cite{Ko}, we denote
\begin{equation}\label{2.3}
(n;k)_{{\bf z},\alpha} := z_n^\alpha\prod_{i=1}^k(z_n-z_i),\quad n\geq 1,\quad
k\geq 0,
\end{equation}
where $n$ and $k$ are integers and $\alpha\geq 0$ is a real constant. We then
define the zeta function belonging to the sequence $\bf z$ by
\begin{equation}\label{2.4}
\zeta_{\bf z}(s) := \sum_{n=1}^\infty\frac{1}{z_n^s},\quad {\rm Re}(s)>1.
\end{equation}
Absolute convergence of this series for ${\rm Re}(s)>1$ is guaranteed by the
condition \eqref{2.2}. It is now easy to derive an analogue of the generating
function \eqref{1.2}.

\begin{lemma}\label{lem:2.1}
Let $\bf z$ be a sequence satisfying \eqref{2.1} and \eqref{2.2}, and 
$\alpha\in{\mathbb R}$ be such that either $\alpha>0$, or $\alpha=0$ and 
$\sum z_n^{-1}$ converges. Then
\begin{equation}\label{2.5}
\sum_{n=1}^\infty\frac{1}{(z_n-x)z_n^\alpha} 
= \sum_{k=0}^\infty\zeta_{\bf z}(k+\alpha+1)x^k,\quad |x| <\min\{1, z_1\}.
\end{equation}
\end{lemma}

\begin{proof}
By the condition on $\alpha$, and since $|x|<\min\{1, z_1\}$, both sides of 
\eqref{2.5} are 
absolutely convergent, and the following series operations are legitimate.
Beginning with the left-hand side and using geometric series expansions, we get
\begin{align*}
\sum_{n=1}^\infty\frac{1}{(z_n-x)z_n^\alpha}
&=\sum_{n=1}^\infty\frac{1}{z_n^{\alpha+1}(1-\frac{x}{z_n})}
=\sum_{n=1}^\infty\frac{1}{z_n^{\alpha+1}}\sum_{k=0}^\infty\frac{x^k}{z_n^k}\\
&=\sum_{k=0}^\infty\bigg(\sum_{n=1}^\infty\frac{1}{z_n^{k+\alpha+1}}\bigg)x^k.
\end{align*}
This, with \eqref{2.4}, gives \eqref{2.5} as required.
\end{proof}

We are now ready to state and prove our main result. The proof follows 
Koecher's ideas in \cite{Ko}.

\begin{theorem}\label{thm:2.2}
Let $\bf z$ and $\alpha$ be as in Lemma~\ref{lem:2.1}, and suppose that
\begin{equation}\label{2.5a} 
P_N:=\frac{z_1+z_N}{z_{N+1}-z_N}\cdot\frac{z_1+z_{N-1}}{z_{N+1}-z_{N-1}}\cdots
\frac{z_1+z_1}{z_{N+1}-z_1},\qquad N=1, 2,\ldots,
\end{equation}
is a bounded sequence. If we set 
\begin{equation}\label{2.6}
\gamma_k(x) := \frac{1}{z_k-x}\cdot\frac{1}{(k;k-1)_{{\bf z},\alpha}}
+ \sum_{n=k+1}^\infty\frac{1}{(n;k)_{{\bf z},\alpha}},
\end{equation}
then for all complex $x$ with $|x|<z_1$ we have
\begin{equation}\label{2.7}
\sum_{n=1}^\infty\frac{1}{(z_n-x)z_n^\alpha}
= \sum_{k=1}^\infty\gamma_k(x)\prod_{\ell=1}^{k-1}(x-z_\ell).
\end{equation}
\end{theorem}

\begin{proof}
For greater ease of notation, we suppress the subscripts ${\bf z}, \alpha$.
We consider the sequence of functions $\varphi_k(x)$, defined by the series
\begin{equation}\label{2.8}
\varphi_k(x) := \sum_{n=k+1}^\infty\frac{1}{(z_n-x)(n;k)},\quad k=0, 1, 2,\ldots
\end{equation}
Since $(n;0)=z_n^\alpha$ by \eqref{2.3}, we have
\begin{equation}\label{2.9}
\varphi_0(x) := \sum_{n=1}^\infty\frac{1}{(z_n-x)z_n^\alpha},
\end{equation}
which is the left-hand side of \eqref{2.7}. Now we claim that the sequence
$\varphi_k(x)$ satisfies the recurrence relation
\begin{equation}\label{2.10}
\varphi_{k-1}(x)-(x-z_k)\varphi_k(x) = \frac{1}{(z_k-x)(k;k-1)}
+\sum_{n=k+1}^\infty\frac{1}{(n;k)},\quad k\geq 1.
\end{equation}
To prove this identity, we first note that from \eqref{2.3} we get
\begin{equation}\label{2.11}
\frac{1}{(n;k-1)} = \frac{z_n-z_k}{(n;k)}.
\end{equation}
Now we have by \eqref{2.8},
\begin{align*}
\varphi_{k-1}(x)&-(x-z_k)\varphi_k(x) 
=\sum_{n=k}^\infty\frac{1}{(z_n-x)(n;k-1)}
+(z_k-x)\sum_{n=k+1}^\infty\frac{1}{(z_n-x)(n;k)}\\
&= \frac{1}{(z_k-x)(k;k-1)}+\sum_{n=k+1}^\infty\frac{1}{(z_n-x)(n;k-1)}\\
&\quad+(z_k-x)\sum_{n=k+1}^\infty\frac{1}{(z_n-x)(n;k)}\\
&= \frac{1}{(z_k-x)(k;k-1)}
+\sum_{n=k+1}^\infty\frac{1}{(z_n-x)}\bigg(\frac{1}{(n;k-1)}
+\frac{z_k-x}{(n;k)}\bigg).
\end{align*}
With \eqref{2.11} we get
\[
\frac{1}{(n;k-1)}+\frac{z_k-x}{(n;k)} = \frac{z_n-z_k+z_k-x}{(n;k)}
= \frac{z_n-x}{(n;k)},
\]
which then leads to \eqref{2.10}, as claimed.

To solve the recurrence \eqref{2.10}, we denote its right-hand side by 
$\gamma_k(x)$, as in \eqref{2.6}. For each $k\geq 2$, we multiply both sides of
\eqref{2.10} by $(x-z_1)\cdots(x-z_{k-1})$ and add the resulting identities for
$k=1,2,\ldots,N$, for some positive integer $N$. This is a telescoping sum,
giving
\begin{equation}\label{2.12}
\varphi_0(x)-(x-z_1)\cdots(x-z_N)\varphi_N(x)
= \sum_{k=1}^N\gamma_k(x)\prod_{\ell=1}^{k-1}(x-z_\ell).
\end{equation}
With \eqref{2.8} and \eqref{2.3} we get
\begin{equation}\label{2.13}
(x-z_1)\cdots(x-z_N)\varphi_N(x) 
= \sum_{n=N+1}^{\infty}\frac{1}{(z_n-x)z_n^\alpha}\cdot
\frac{(x-z_1)\cdots(x-z_N)}{(z_n-z_1)\cdots(z_n-z_N)}.
\end{equation}
Using the fact that $|x|<z_1$, we get the estimate
\[
\left|\frac{(x-z_1)\cdots(x-z_N)}{(z_n-z_1)\cdots(z_n-z_N)}\right|
\leq \frac{(z_1+z_1)\cdots(z_1+z_N)}{(z_{N+1}-z_1)\cdots(z_{N+1}-z_N)},
\]
where in the denominator we have used the fact that $z_n\geq z_{N+1}$.
By the condition \eqref{2.5a}, the large fraction on the right 
of \eqref{2.13} is bounded, and due to the
conditions on $\alpha$, the series on the right converges.
Therefore
\[
\lim_{N\rightarrow\infty}(x-z_1)\cdots(x-z_N)\varphi_N(x) = 0,
\]
and this means that \eqref{2.12}, with \eqref{2.9}, implies \eqref{2.7} as
$N\rightarrow\infty$. The proof is now complete.
\end{proof}

\section{Two particular cases}

In this paper we will mainly consider two instances of Theorem~\ref{thm:2.2},
which we state as further theorems. The first one will have Koecher's original
result as a special case.

\begin{theorem}\label{thm:2.3}
Let the sequence ${\bf z}$ be given by
\begin{equation}\label{2.14}
z_n = (n+c)^\beta+d,\qquad c>-1,\; d\geq 0,\; \beta>1,
\end{equation}
for all integers $n\geq 1$. If $\alpha\geq 0$ and $\gamma_k(x)$ is as defined 
in \eqref{2.6}, then for all complex $x$ with $|x|<\min\{1,z_1\}$ we have
\begin{equation}\label{2.15}
\sum_{m=0}^\infty\zeta_{\bf z}(m+\alpha+1)x^m
= \sum_{k=1}^\infty\gamma_k(x)\prod_{\ell=1}^{k-1}(x-z_\ell).
\end{equation}
\end{theorem}

\begin{proof}
The sequence ${\bf z}$ in \eqref{2.14} clearly satisfies \eqref{2.1} and
\eqref{2.2}. It remains to verify that the sequence $P_N$ in \eqref{2.5a} is
bounded. We first consider the denominator and note that
\[
z_{N+1}-z_{N+1-j}=(N+1+c)^\beta-(N+1+c-j)^\beta,\qquad j=1,2,\ldots,N.
\]
The Mean Value Theorem gives us
\[
\frac{(N+1+c)^\beta-(N+1+c-j)^\beta}{j} = \beta x_j^{\beta-1},\qquad
N+1+c-j < x_j < N+1+c,
\]
so that
\[
z_{N+1}-z_{N+1-j} > j\,\beta(N+1+c-j)^{\beta-1},\qquad j=1,2,\ldots,N.
\]
Hence the denominator of \eqref{2.5a} is larger than
\begin{align*}
N!\beta^N&\big((N+c)(N-1+c)\cdots(1+c)\big)^{\beta-1} \\
&=\frac{\Gamma(N+1)\Gamma(c+1)\beta^N}{\Gamma(N+1+c)}
\big((N+c)(N-1+c)\cdots(1+c)\big)^{\beta},
\end{align*}
and with \eqref{2.14} and \eqref{2.5a} we get, with $\tilde{d}:=d+z_1$,
\begin{equation}\label{2.16}
P_N < \frac{1}{\Gamma(c+1)}\cdot\frac{\Gamma(N+1+c)}{\Gamma(N+1)\beta^N}
\cdot\frac{(N+c)^\beta+\tilde{d}}{(N+c)^\beta}
\cdot\frac{(N-1+c)^\beta+\tilde{d}}{(N-1+c)^\beta}
\cdots\frac{(1+c)^\beta+\tilde{d}}{(1+c)^\beta}.
\end{equation}
We now consider two different limits for the right-hand side of \eqref{2.16}.
First,
\begin{align*}
\lim_{N\to\infty}\frac{\Gamma(N+1+c)}{\Gamma(N+1)\beta^N}
&= \lim_{N\to\infty}\left(\frac{\Gamma(N+c)}{\Gamma(N)\cdot N^c}
\cdot\frac{N^c}{\beta^{N-1}}\right) \\
&= \lim_{N\to\infty}\frac{\Gamma(N+c)}{\Gamma(N)\cdot N^c}
\cdot\lim_{N\to\infty}\frac{N^c}{\beta^{N-1}} = 0
\end{align*}
since, by a well-known property of the gamma function, the first limit on the
right-hand side is 1, while the second limit is 0 for any real $c>-1$ since
$\beta>1$. 

Next, the remaining fractions on the right of \eqref{2.16} can be rewritten,
in reverse order, as
\[
\bigg(1+\frac{\tilde{d}}{(1+c)^\beta}\bigg)
\bigg(1+\frac{\tilde{d}}{(2+c)^\beta}\bigg)\cdots
\bigg(1+\frac{\tilde{d}}{(N+c)^\beta}\bigg).
\]
The limit of this expression, as $N\to\infty$, is an infinite product which 
converges since 
\[
\sum_{j=1}^\infty\frac{\tilde{d}}{(j+c)^\beta},\qquad \beta>1, c>-1,
\]
is a convergent series. Hence the limit of the right of \eqref{2.16} is 0,
as $N\to\infty$. This, in turn means that $P_N$ is a bounded sequence.
All conditions of Theorem~\ref{thm:2.2} are therefore satisfied, and 
\eqref{2.15} follows from \eqref{2.5} and \eqref{2.7}.
\end{proof}

For applying Theorem~\ref{thm:2.2} or Theorem~\ref{thm:2.3}, the infinite 
series on the right of \eqref{2.6} is usually the most difficult part to 
evaluate. However, the evaluation required for deriving Koecher's result 
\eqref{1.3} is quite straightforward, and is given in the following lemma. We 
state a more general version, which will be useful later in this paper.

\begin{lemma}\label{lem:2.4}
Let $(n;k)=(n+r+k)(n+r+k-1)\cdots(n+r-k)$, where $r\geq 0$ is an integer. Then for each $k\geq 1$ we have
\begin{equation}\label{2.17}
\sum_{n=k+1}^\infty\frac{1}{(n;k)} = \frac{r!}{2k(2k+r)!}.
\end{equation}
\end{lemma}

\begin{proof}
The following partial fraction expansion is easy to find and to verify:
\begin{align*}
&\frac{1}{(n+r+k)(n+r+k-1)\cdots(n+r-k)} \\
&\qquad =\frac{1/2k}{(n+r+k-1)\cdots(n+r-k)}
-\frac{1/2k}{(n+r+k)\cdots(n+r-k+1)}.
\end{align*}
Summing this over all $n\geq k+1$ gives a telescoping series on the right, with
only the first term on the right remaining for $n=k+1$. That is, we are left 
with 
\[
\frac{1}{2k}\cdot\frac{1}{(2k+r)(2k+r-1)\cdots r+1} = \frac{r!}{2k(2k+r)!},
\]
as claimed.
\end{proof}

To obtain Koecher's identity \eqref{1.3}, we set $c=d=0$ and $\beta=2$ in
\eqref{2.14}, so that $z_n=n^2$. Furthermore, we replace $x$ by $x^2$ and set 
$\alpha=1/2$; then by \eqref{2.4} with $r=0$ the left-hand side of \eqref{2.15}
becomes the left-hand side of \eqref{1.2}. Next, by \eqref{2.3} we have
\[
(n;k)_{{\bf z},1/2} = (n+k)\cdots(n+1)n(n-1)\cdots(n-k)
\] 
and $(k;k-1)_{{\bf z},1/2}=(2k-1)!$. Hence with \eqref{2.6} and \eqref{2.17} 
we have
\[
\gamma_k(x^2) = \frac{1}{k^2-x^2}\cdot\frac{1}{(2k-1)!}+\frac{1}{2k(2k)!}
= \frac{1}{2k(2k)!}\frac{5k^2-x^2}{k^2-x^2},
\]
and the right-hand side of \eqref{2.15} becomes
\[
\sum_{k=0}^\infty\frac{1}{2k(2k)!}\frac{5k^2-x^2}{k^2-x^2}(-1)^{k-1}(k-1)!^2
\prod_{\ell=1}^{k-1}\bigg(1-\frac{x^2}{\ell^2}\bigg).
\]
This is the same as the right-hand side of \eqref{1.3}, as claimed.

\medskip
Since the proof of Theorem~\ref{thm:2.3} required $\beta > 1$, the situation in
the case $\beta=1$ in \eqref{2.14} will be quite different. The next result 
deals with this case.

\begin{theorem}\label{thm:2.5}
Let the sequence ${\bf z}$ be given by $z_n = n+c$ for all integers $n\geq 1$,
where $c>-1$ is a real constant. If $\alpha>\max\{0,c\}$ and $\gamma_k(x)$ is 
as defined in \eqref{2.6}, then for all complex $x$ with $|x|\leq\varepsilon$,
where $\varepsilon>0$ is sufficiently small, we have
\begin{equation}\label{2.18}
\sum_{m=0}^\infty\zeta_{\bf z}(m+\alpha+1)x^m
= \sum_{k=1}^\infty\gamma_k(x)\prod_{\ell=1}^{k-1}(x-z_\ell).
\end{equation}
\end{theorem}

\begin{proof}
The sequence ${\bf z}$ once again satisfies \eqref{2.1} and \eqref{2.2}. By
the proof of Theorem~\ref{thm:2.2} we are done if we can show that the 
right-hand side of \eqref{2.13} approaches 0 as $N\to\infty$. To do so, we
set $z_n=n+c$, and using $|x|\leq\varepsilon$ and $n\geq N+1$, we get
\begin{align}
\left|\frac{(x-z_1)\cdots(x-z_N)}{(z_n-z_1)\cdots(z_n-z_N)}\right|
& \leq \frac{(1+c+\varepsilon)\cdots(N+c+\varepsilon)}{N!}\label{2.19} \\
& =\frac{\Gamma(N+1+c+\varepsilon)}{\Gamma(N+1)\Gamma(1+c+\varepsilon)}.\nonumber
\end{align}
Since
\[
\lim_{N\to\infty}\frac{\Gamma(N+1+c+\varepsilon)}{\Gamma(N+1)(N+1)^{c+\varepsilon}} = 1
\]
by a well-known property of the gamma function, we have
\begin{equation}\label{2.20}
\frac{\Gamma(N+1+c+\varepsilon)}{\Gamma(N+1)\Gamma(1+c+\varepsilon)}
\sim \frac{1}{\Gamma(1+c+\varepsilon)}\cdot N^{c+\varepsilon}.
\end{equation}
On the other hand,
\[
\sum_{n=N+1}^{\infty}\frac{1}{(z_n-x)z_n^\alpha}
\leq \sum_{n=N+1}^{\infty}\frac{1}{(n+c-\varepsilon)(n+c)^\alpha}
\leq \delta\sum_{n=N+1}^{\infty}\frac{1}{(n+c)^{\alpha+1}},
\]
for some constant $\delta>1$. Now, using an integral estimate, we get for
$\alpha>0$,
\begin{equation}\label{2.21}
\sum_{n=N+1}^{\infty}\frac{1}{(z_n-x)z_n^\alpha}
< \delta\int_N^\infty\frac{dx}{(x+c)^{\alpha+1}}
= \frac{\delta}{\alpha}\cdot\frac{1}{(N+c)^\alpha}.
\end{equation}
Combining \eqref{2.19}--\eqref{2.21}, we see that the right-hand side of 
\eqref{2.13} is dominated by a sequence that is asymptotically equivalent to
\[
\frac{\delta}{\alpha\Gamma(1+c+\varepsilon)}\cdot N^{c+\varepsilon-\alpha}.
\]
Hence the right-hand side of \eqref{2.13} approaches 0 as $N\to\infty$
when $\alpha>c+\varepsilon$. This is the case whenever $\alpha>c$ and
$\varepsilon$ is chosen sufficiently small, for instance 
$\varepsilon=(\alpha-c)/2$. The proof of Theorem~\ref{thm:2.5} is now complete.
\end{proof}

\section{Results on Euler sums}

One of the first and most famous result linking a Riemann zeta value with 
values of multiple zeta functions is Euler's identity
\[
\sum_{n=1}^\infty\frac{1}{n^3} 
= 8\sum_{n=1}^\infty\frac{(-1)^n}{n^2}\sum_{j=1}^{n-1}\frac{1}{j},
\]
which can be written as
\begin{equation}\label{4.1}
\zeta(3) = 8\zeta(\overline{2},1);
\end{equation}
see, e.g., \cite{BB}. In \eqref{4.1} and in what follows we use the standard
notation 
\begin{equation}\label{4.2}
\zeta(s_1,\ldots,s_m) 
= \sum_{k_1>\cdots>k_m\geq 1}\frac{1}{k_1^{s_1}}\cdots\frac{1}{k_m^{s_m}},
\end{equation}
with the further convention that an overlined variable, e.g., $\overline{s}_j$,
indicates that $1/k_j^{s_j}$ is replaced by $(-1)^{k_j}/k_j^{s_j}$ in 
\eqref{4.2}. It is also customary to denote $r$ repetitions of a substring $S$
by $\{S\}^r$, with $\{S\}^0$ being the empty string. Thus, for example,
\[
\zeta(\overline{2},\{1\}^3) 
= \sum_{k_1>\cdots>k_4\geq 1}\frac{(-1)^{k_1}}{k_1^2k_2k_3k_4}. 
\]
We are now ready to state and prove the main result of this section.

\begin{theorem}\label{thm:4.1}
For any integer $n\geq 2$ we have
\begin{equation}\label{4.3}
\zeta(n) = (-2)^{n-1}\bigg(2\zeta(\overline{2},\{1\}^{n-2})
+\sum_{j=3}^{n-1}(-1)^j\zeta(\overline{j},\{1\}^{n-j})\bigg).
\end{equation}
\end{theorem}

Before proving this identity, we note that for $n=2$ it is trivially true, 
and for $n=3, 4, 5$ we get
\begin{align*}
\zeta(3) &= 8\zeta(\overline{2},1),\\
\zeta(4) &= -16\zeta(\overline{2},1,1)+8\zeta(\overline{3},1),\\
\zeta(5) &= 32\zeta(\overline{2},1,1,1)-16\zeta(\overline{3},1,1)
+16\zeta(\overline{4},1).
\end{align*}
The identity \eqref{4.3} could therefore be considered a generalization of
Euler's identity \eqref{4.1}.

\begin{proof}[Proof of Theorem~\ref{thm:4.1}]
We use Theorem~\ref{thm:2.5} with $c=0$ and $\alpha=1$. Then $\zeta(k+2)$ will
be the coefficient of $x^k$ on the right-hand side of \eqref{2.18}.

To evaluate $\gamma_k(x)$ in \eqref{2.6}, we first note that by \eqref{2.3}
we have 
\[
(n;k)_{{\bf z},1}=n(n-1)\cdots(n-k), 
\]
and in particular
$(k;k-1)_{{\bf z},1}=k!$. Furthermore, with a partial fraction expansion similar
to the one in the proof of Lemma~\ref{lem:2.4}, namely
\[
\frac{1}{n(n-1)\cdots(n-k)}
=\frac{1/k}{(n-1)\cdots(n-k)}-\frac{1/k}{n\cdots(n-k+1)},
\]
we get a telescoping series which sums as
\[
\sum_{n=k+1}^\infty\frac{1}{(n;k)_{{\bf z},1}} = \frac{1}{k\cdot k!}.
\]
Hence
\begin{align}
\gamma_k(x) &= \frac{1}{k-x}\cdot\frac{1}{k!}+\frac{1}{k\cdot k!}
= \frac{1}{k\cdot k!}\bigg(\frac{1}{1-\frac{x}{k}}+1\bigg)\label{4.4}\\
&=\frac{1}{k\cdot k!}\bigg(2+\sum_{j=1}^\infty\frac{x^j}{k^j}\bigg).\nonumber
\end{align}
To deal with the right-hand side of \eqref{2.18}, which we call $R(x)$, we 
use the finite multiple sums 
\[
H_K(m):=\sum_{K\geq k_1>\cdots>k_m\geq 1}\frac{1}{k_1\cdots k_m}\qquad(m\geq 1),
\]
which are sometimes called hyperharmonic numbers, with the convention 
$H_K(0)=1$ for all integers $K\geq 1$. Note that $H_K(1)=H_K$, the $K$th 
harmonic number. With this notation we have
\begin{align*}
\prod_{\ell=1}^{k-1}(x-z_\ell) &= \prod_{\ell=1}^{k-1}(x-\ell)
= (-1)^{k-1}(k-1)!\prod_{\ell=1}^{k-1}\left(1-\tfrac{x}{\ell}\right)\\
&= (-1)^{k-1}(k-1)!\sum_{\nu=0}^{k-1}(-1)^\nu H_{k-1}(\nu)x^\nu.
\end{align*}
Upon multiplying this by \eqref{4.4}, we get
\begin{align*}
R(x) &= \sum_{k=1}^\infty\frac{(-1)^{k-1}}{k^2}
\bigg(2+\sum_{j=1}^\infty\frac{x^j}{k^j}\bigg)
\bigg(\sum_{\nu=0}^{k-1}(-1)^\nu H_{k-1}(\nu)x^\nu\bigg)\\
 &= \sum_{k=1}^\infty\frac{(-1)^{k-1}}{k^2}S_k(x),
\end{align*}
where
\[
S_k(x) = 2\sum_{\nu=0}^{k-1}(-1)^\nu H_{k-1}(\nu)x^\nu
+\sum_{\mu=1}^\infty
\bigg(\sum_{j=1}^\mu\frac{(-1)^{\mu-j}}{k^j}H_{k-1}(\mu-j)\bigg)x^\mu.
\]
Hence, by equating coefficients of powers of $x$ on both sides of \eqref{2.18},
we get
%\begin{equation}\label{4.5}
%\zeta(\mu+2) = \sum_{k=1}^\infty\frac{(-1)^{k-1}}{k^2}T_{k,\mu},
%\end{equation}
%where
\[
\zeta(\mu+2) = \sum_{k=1}^\infty\frac{(-1)^{k-1}}{k^2}
\bigg(2(-1)^\mu H_{k-1}(\mu)
+\sum_{j=1}^{\mu-1}\frac{(-1)^{\mu-j}}{k^j}H_{k-1}(\mu-j)+\frac{1}{k^\mu}\bigg).
\]
%\[
%T_{k,\mu}=2(-1)^\mu H_{k-1}(\mu)
%+\sum_{j=1}^{\mu-1}\frac{(-1)^{\mu-j}}{k^j}H_{k-1}(\mu-j)+\frac{1}{k^\mu}.
%\]
So with this and the definition of Euler sums we get
\begin{align}
\zeta(\mu+2) &= 2(-1)^{\mu-1}\zeta(\overline{2},\{1\}^\mu)\label{4.5}\\
&\quad+\sum_{j=1}^{\mu-1}(-1)^{\mu-j-1}\zeta(\overline{j+2},\{1\}^{\mu-j})
+\sum_{k=1}^\infty\frac{(-1)^{k-1}}{k^{\mu+2}}.\nonumber
\end{align}
The last sum on the right of \eqref{4.5} can be rewritten as
\[
\sum_{k=1}^\infty\frac{(-1)^{k-1}}{k^{\mu+2}}
=\sum_{k=1}^\infty\frac{1}{k^{\mu+2}}-2\sum_{k=1}^\infty\frac{1}{(2k)^{\mu+2}}
= \bigg(1-\frac{1}{2^{\mu+1}}\bigg)\zeta(\mu+2).
\]
Combining this with \eqref{4.5}, we get
\begin{equation}\label{4.6}
\zeta(\mu+2) = (-1)^{\mu-1}2^{\mu+2}\zeta(\overline{2},\{1\}^\mu)
+2^{\mu+1}\sum_{j=1}^{\mu-1}(-1)^{\mu-j-1}\zeta(\overline{j+2},\{1\}^{\mu-j}).
\end{equation}
Finally, we replace $\mu$ by $n-2$ ($n\geq 2$) and shift the summation on the
right of \eqref{4.6} by 2. Then \eqref{4.6} immediately gives \eqref{4.3}, and
the proof is complete.
\end{proof}

We now present an alternative proof of Theorem~\ref{thm:4.1}, based on a
recent result of Xu \cite[Theorem~2.1]{Xu} who showed that
\begin{equation}\label{4.8}
\zeta\big(\overline{k+2},\{1\}^{m-1}\big)=\frac{(-1)^{m+k}}{m!k!}
\int_0^1(\log x)^k(\log(1+x))^{m}\frac{dx}{x}.
\end{equation}
For $n\geq 2$ we denote by $S_n$ the expression in large parentheses in 
\eqref{4.3}, namely
\begin{equation}\label{4.9}
S_n := 2\zeta(\overline{2},\{1\}^{n-2})
+\sum_{j=3}^{n-1}(-1)^j\zeta(\overline{j},\{1\}^{n-j}),
\end{equation}
and we prove the following result.

\begin{theorem}\label{thm:4.2}
We have the generating function
\begin{equation}\label{4.10}
H(z) := \sum_{n=2}^\infty S_nz^{n-1} = \psi(1)-\psi(1+\tfrac{z}{2}),
\end{equation}
where $\psi(z)=\Gamma'(z)/\Gamma(z)$ is the digamma function.
\end{theorem}

Before we prove this, we note that a well-known generating function for the
Riemann zeta function implies
\begin{equation}\label{4.11}
\sum_{n=2}^\infty\zeta(n)(-\tfrac{z}{2})^{n-1} = \psi(1)-\psi(1+\tfrac{z}{2});
\end{equation}
see, e.g., \cite[Eq.~25.8.5]{DLMF}. Comparing \eqref{4.10} with \eqref{4.11},
we immediately obtain Theorem~\ref{thm:4.1}. We mention in passing that
$\psi(1)=-\gamma$, where $\gamma$ is the Euler-Mascheroni constant.

For the proof of Theorem~\ref{thm:4.2} we require two lemmas; the first one
is the evaluation of an integral which we were unable to find in the literature.

\begin{lemma}\label{lem:4.3}
For any $z\in{\mathbb R}$ with $z>0$ we have
\begin{equation}\label{4.12}
\int_0^1\left(\frac{1+x^z}{(1+x)^z}-1\right)\frac{dx}{x}
= \psi(1)-\psi(z).
\end{equation}
\end{lemma}

\begin{proof}
We denote the integral on the left of \eqref{4.12} by $I(z)$. Then it is easy 
to verify that 
\begin{equation}\label{4.13}
I(z+1)-I(z) = - \int_0^1\frac{1+x^{z-1}}{(1+x)^{z+1}}dx = -\frac{1}{z},
\end{equation}
where the second equality comes from a special case of the identity 2.2.2.14
in \cite{PrE}.

We now compare \eqref{4.13} with the well-known functional equation 
\begin{equation}\label{4.14}
\psi(z+1)-\psi(z) = \frac{1}{z},
\end{equation}
(see, e.g., \cite[Eq.~5.5.2]{DLMF}) and apply the following uniqueness result:

{\it Suppose that $\varphi$ is defined on $[N,\infty)$ for some non-negative
integer $N$ and that $\varphi\to 0$ as $x\to\infty$. Then if there exist two
nondecreasing solutions of the equation $f(z+1)-f(z) = \varphi(x)$, they must
differ by a constant.}

This can be found in \cite[Lemma~1.1]{Mu}, where further references are given.
To apply this result, we set $f(z)=-I(z)$ and note that for $z>0$ and $x>0$,
\[
\frac{d}{dz}\left(\frac{1+x^z}{(1+x)^z}-1\right)
=\frac{x^z}{(1+x)^z}\big(\log{x}-\log(1+x)\big)-\frac{\log(1+x)}{(1+x)^z} < 0.
\]
Hence $f(z)$ is nondecreasing, and comparing \eqref{4.13} with \eqref{4.14} we
therefore have $f(z)=-I(z)=\psi(z)+C$ for some constant $C$. We determine $C$
by setting $z=1$; since $I(1)=0$, we have $C=-\psi(1)$, which completes the
proof. 
\end{proof}

\begin{lemma}\label{lem:4.4}
For any integer $n\geq 2$ we have
\begin{align}
S_n = \frac{1}{(n-1)!}\int_{0}^{1}\bigg(&\left(\log\frac{1}{1+x}\right)^{n-1}
-(\log x)^{n-1}+\left(\log\frac{x}{1+x}\right)^{n-1}\label{4.15}\\
&+(n-1)(\log x)^{n-2}\log(1+x)\bigg)\frac{dx}{x}.\nonumber
\end{align}
\end{lemma}

\begin{proof}
Substituting \eqref{4.8} into \eqref{4.9}, we get
\begin{align}
S_n &= \frac{(-1)^{n-1}}{(n-1)!}\bigg(2\int_0^1\big(\log(1+x)\big)^{n-1}\frac{dx}{x}\label{4.16}\\
&+\int_0^1\bigg(\sum_{j=3}^{n-1}(-1)^j\binom{n-1}{j-2}(\log{x})^{j-2}
\big(\log(1+x)\big)^{n-j+1}\bigg)\bigg)\frac{dx}{x}.\nonumber
\end{align}
The sum in the second term is now seen to be
\begin{align*}
\sum_{j=1}^{n-3}&(-1)^j\binom{n-1}{j}(\log{x})^j\left(\log(1+x)\right)^{n-1-j}
= \big(\log(1+x)-\log{x}\big)^{n-1} \\
&-\log(1+x)^{n-1}-(-1)^n\left((n-1)(\log{x})^{n-2}\log(1+x)-(\log{x})^{n-1}\right),
\end{align*}
and combining this with \eqref{4.16} we obtain \eqref{4.15}.
\end{proof}

We are now ready to prove Theorem~\ref{thm:4.2}.

\begin{proof}[Proof of Theorem~\ref{thm:4.2}]
We substitute \eqref{4.15} into the infinite series \eqref{4.10} and interchange
the order of summation and integration, which is easy to justify. Using the
evaluations
\begin{align*}
\sum_{n=2}^\infty\frac{z^{n-1}}{(n-1)!}\left(\log\frac{1}{1+x}\right)^{n-1}
&=\frac{1}{\left(1+x\right)^{z}}-1,\\
\sum_{n=2}^\infty\frac{z^{n-1}}{(n-1)!}\left(\log x\right)^{n-1} & =x^{z}-1,\\
\sum_{n=2}^\infty\frac{z^{n-1}}{(n-1)!}\left(\log\frac{x}{1+x}\right)^{n-1}
&=\left(\frac{x}{1+x}\right)^{z}-1,
\end{align*}
and
\[
\sum_{n=2}^\infty\frac{z^{n-1}}{(n-1)!}(n-1)(\log x)^{n-2}\log(1+x)
=z\log(1+x)x^{z},
\]
we find
\begin{equation}\label{4.17}
H(z)=\int_{0}^{1}\left(\frac{1+x^{z}}{(1+x)^z}-1+x^z\big(z\log(1+x)-1\big)\right)\frac{dx}{x}.
\end{equation}
Splitting this integral into three parts, we use Lemma~\ref{lem:4.3} as well as
the integral
\[
\int_0^1 x^{z-1}\log(1+x)dx 
=\frac{\log{2}}{z}-\frac{1}{2z}\left(\psi(\tfrac{z+2}{2})-\psi(\tfrac{z+1}{2})\right)
\]
(see \cite[Eq.~4.293.1]{GR}) and the obvious integral
\[
\int_0^1 x^{z-1}dx = \frac{1}{z}.
\]
Substituting all this into \eqref{4.17}, we get
\[
H(z) = \psi(1)-\psi(z)+\log{2}-\tfrac{1}{2}\psi(\tfrac{z}{2}+1)
+\tfrac{1}{2}\psi(\tfrac{z+1}{2})-\tfrac{1}{z}.
\]
Finally, using the identities \eqref{4.14} and 
$\psi(2y)=\frac{1}{2}\big(\psi(y)+\psi(y+\frac{1}{2})\big)+\log{2}$
(see \cite[Eq.~5.5.8]{DLMF}) with $y=(z+1)/2$, we get
\[
H(z) = 
\psi(1)-\psi(z+1)+\left(\psi(z+1)-2\cdot\tfrac{1}{2}\psi(\tfrac{z}{2}+1)\right),
\]
which is the right-hand side of \eqref{4.10}, as required.
\end{proof}

%!!
\section{Generalizing the Markov-Ap\'ery identity}

We recall that Koecher's identity \eqref{1.3} was obtained from 
Theorem~\ref{thm:2.3} as a special case by taking $z_n=n^2$. The identity
\eqref{1.1} then followed by setting $x=0$. In this section we will again use
Theorem~\ref{thm:2.3}, with $x=0$ from the beginning but with $z_n=(n+c)^2$
for integers $c\geq 0$. For the remainder of this section we set
\begin{equation}\label{5.0}
\zeta_c(s):= \zeta(s)-\sum_{j=1}^c\frac{1}{j^s},\qquad c=0,1,2,\ldots;
\end{equation}
in particular, $\zeta_0(s)=\zeta(s)$ and $\zeta_1(s)=\zeta(s)-1$.
Also, $\zeta_c(s)$ is in fact the Hurwitz zeta function $\zeta(s,c)$.
We are now ready to state the main result of this section.

\begin{theorem}\label{thm:5.1}
If $c\geq 0$ is an integer, then
\begin{equation}\label{5.1}
\zeta_c(3) = \frac{1}{2c!^2}\sum_{k=1}^\infty
\frac{(-1)^{k-1}}{\binom{2k+2c}{k+c}}\cdot\frac{P_c(k)}{(k+c)^2k(k+1)\cdots(k+c)},
\end{equation}
%where $P_c(k)$ is a polynomial in $k$ of degree $3c$, with leading coefficient
%5 and, for $c\geq 1$, constant coefficient $c!(2c)!$.
%Furthermore, the polynomials $P_c(k)$ are explicitly given by
where
\begin{align}
P_c(k)&=\frac{4(k+2c)!(k+c-1)!}{(k+c)(k-1)!^2}
+\frac{2(k+c)!(2k+2c)!}{(k-1)!}\sum_{j=0}^{2c}\frac{(2c-j)!}{j!(2k+2c-j)!}\label{5.2}\\
&\quad\times\sum_{\nu=0}^j\frac{(-1)^\nu}{k+c-\nu}\binom{j}{\nu}
\frac{(2k+j-1-\nu)!(k+2c-\nu)!}{(k-1-\nu)!(2k+2c-\nu)!}.\nonumber
\end{align}
\end{theorem}

\noindent
{\bf Remark.} We can use \eqref{5.2} to compute
\begin{align*}
P_0(k)&=5,\\
P_1(k)&=5k^3+12k^2+4k+2,\\
P_2(k)&=5k^6+49k^5+171k^4+271k^3+232k^2+128k+48,\\
P_3(k)&=5k^9+111k^8+1011k^7+4935k^6+14262k^5+25734k^4\\
&\quad+30190k^3+24048k^2+13248k+4320,\\
P_4(k)&=5k^{12}+198k^{11}+3409k^{10}+33650k^9+211731k^8+894834k^7\\
&\quad+2613523k^6+5362734k^5 + 7817348k^4 + 8176552k^3 + 6167424k^2\\
&\quad+3244032k+967680,\\
P_5(k)&=5k^{15}+310k^{14}+8625k^{13}+142600k^{12}+1564435k^{11}+12049820k^{10}\\
&\quad+67279375k^9+277409600k^8+853390140k^7+1968104030k^6\\
&\quad+3407457500k^5+4426865800k^4+4304943120k^3+3095389440k^2\\
&\quad+1556582400k+435456000.
\end{align*}
Based on these evaluations and on further computations for several $c\geq 6$, 
we conjecture that all $P_c(k)$ are polynomials in $k$ of degree $3c$, with 
integer coefficients, including leading coefficient 5 and, for $c\geq 1$, 
constant coefficient $c!(2c)!$.

\medskip
Before proving Theorem~\ref{thm:5.1}, we note that, since $P_0(k)=5$, 
for $c=0$ the identity \eqref{5.1} reduces to \eqref{1.1}. Similarly, for 
$c=1$ and 2 and using $P_1(k)$ and $P_2(k)$ as above, we get after some minor 
manipulations,
\begin{align}
\zeta(3) &= 1+\frac{1}{4}\sum_{k=1}^\infty\frac{(-1)^{k-1}}{\binom{2k}{k}}
\cdot\frac{5k^3+12k^2+4k+2}{k(k+1)^2(2k+1)},\label{5.3}\\
\zeta(3) &= 1+\frac{1}{8}+\frac{1}{16}\sum_{k=1}^\infty
\frac{(-1)^{k-1}}{\binom{2k+2}{k+1}}
\cdot\frac{P_2(k)}{k(k+1)(k+2)^2(2k+3)}.\label{5.4}
\end{align}

The main ingredient in the proof of Theorem~\ref{thm:5.1} is the evaluation
of the infinite series on the right of \eqref{2.6}. We state this as a lemma.

\begin{lemma}\label{lem:5.2}
Let $\alpha=\frac{1}{2}$ and let $\bf z$ be the sequence given by $z_n=(n+c)^2$.
Then
\begin{align}
\sum_{n=k+1}^\infty\frac{1}{(n;k)_{{\bf z},1/2}}
&=\sum_{j=0}^{2c}\frac{(2c-j)!}{(2k+2c-j)!}\label{5.5}\\
&\quad\times\sum_{\nu=0}^j\frac{(-1)^\nu}{k+c-\nu}
\cdot\frac{(2k+j-1-\nu)(k+2c-\nu)!}{(j-\nu)!(k-1-\nu)!(2k+2c-\nu)!\nu!}.\nonumber
\end{align}
\end{lemma}

\begin{proof}
We fix the integers $k$ and $c$. By the definition \eqref{2.3} we have
\[
(n;k)_{{\bf z},1/2} = (n+c)\prod_{i=1}^k\left((n+c)^2-(i+c)^2\right),
\]
and writing the factors in descending order, we get
\begin{equation}\label{5.5a}
(n;k)_{{\bf z},1/2} = (n+k+2c)\cdots(n+2c+1)(n+c)(n-1)\cdots(n-k),
\end{equation}
so that
\begin{equation}\label{5.6}
\sum_{n=k+1}^\infty\frac{1}{(n;k)_{{\bf z},1/2}} = \sum_{n=1}^\infty Q_n,
\end{equation}
where
\begin{equation}\label{5.7}
Q_n :=\frac{1}{n+k+c}\cdot
\frac{(n+k+2c)(n+k+2c-1)\cdots(n+k)}{(n+2k+2c)(n+2k+2c-1)\cdots(n+1)n}.
\end{equation}
It is our goal to find a partial fraction expansion of \eqref{5.7} to make it
possible to apply Lemma~\ref{lem:2.4}. For this purpose we consider $n$, for 
the moment, as a variable and set
\begin{equation}\label{5.8}
Q_n = \sum_{j=0}^{2c}\frac{A_j}{(n+2k+2c-j)\cdots(n+2c-j)},
\end{equation}
where the coefficients $A_0,\ldots,A_{2c}$ are to be determined. To do so, we
fix $j$, $0\leq j\leq 2c$, multiply the right-hand sides of \eqref{5.7} and
\eqref{5.8} by $n+2k+2c-j$, and formally set $n=-2k-2c+j$. Then it is 
straightforward to see that the right-hand side of \eqref{5.7} becomes
\[
\frac{(-1)^j}{k+c-j}\cdot\frac{(k+2c-j)!}{(2k+2c-j)!(k-1-j)!j!},
\]
while the right-hand side of \eqref{5.8} will give
\[
\frac{(-1)^jA_0}{j!(2k-j)!}+\frac{(-1)^{j-1}A_1}{(j-1)!(2k-j+1)!}
+\cdots+\frac{A_j}{(2k)!}.
\]
Equating the two gives, for $j=0,1,\ldots, 2c,$
\begin{equation}\label{5.9}
\binom{2k}{j}A_0-\binom{2k}{j-1}A_1+\cdots+(-1)^j\binom{2k}{0}A_j = B_j,
\end{equation}
where
\begin{equation}\label{5.10}
B_j := \frac{1}{k+c-j}\cdot\frac{(k+2c-j)!(2k)!}{(k-1-j)!(2k+2c-j)!j!}.
\end{equation}
The linear system \eqref{5.9} has a $(2c+1)\times(2c+1)$ lower triangular 
matrix $M$ with diagonal $(1,-1,1,-1,\ldots,1)$, and thus determinant $(-1)^c$.
We claim that its inverse is the matrix $\widetilde{M}$, given by the column
vectors
\begin{equation}\label{5.11}
\left(0,\ldots,0,(-1)^\ell\binom{2k-1}{0},(-1)^\ell\binom{2k}{1},\ldots,
(-1)^\ell\binom{2k-1+2c-\ell}{2c-\ell}\right),
\end{equation}
($\ell=0,1,\ldots,2c$) with $\ell$ initial zeros. Thus, the matrix 
$\widetilde{M}$ is also lower triangular. To show that this is indeed the 
inverse of $M$, we take the inner product of the $j$th row of $M$, as given by
the coefficient sequence in \eqref{5.9}, and the $\ell$th column of 
$\widetilde{M}$ in \eqref{5.11}. When $j<\ell$, this is obviously zero. When
$j\geq\ell$, this inner product is
\begin{align*}
(-1)^\ell\sum_{i=\ell}^j&(-1)^i\binom{2k}{j-i}\binom{2k-1+i-\ell}{i-\ell}\\
&=(-1)^{\ell-j}\sum_{s=0}^{j-\ell}(-1)^s\binom{2k}{s}\binom{2k-1+j-\ell-s}{2k-1},
\end{align*}
where we obtained the right-hand side by setting $s=j-i$ and then slightly
manipulating the second binomial coefficient. Now, when $\ell=j$, the 
right-hand side is obviously 1. When $\ell<j$, a known binomial coefficient
identity, namely Equation 4.2.5.50 in \cite[p.~619]{PrE}, shows that the 
right-hand side vanishes. Hence we have shown that $M\widetilde{M}=I$, and thus
$\widetilde{M}=M^{-1}$, as claimed.

Now, solving the system \eqref{5.9} with the help of the matrix $\widetilde{M}$,
we get
\[
A_j = \sum_{\nu=0}^j(-1)^\nu\binom{2k+j-1-\nu}{j-\nu}B_\nu,
\]
and thus, with \eqref{5.10},
\begin{equation}\label{5.12}
A_j = \sum_{\nu=0}^j\frac{(-1)^\nu\cdot 2k}{k+c-j}
\cdot\frac{(2k+j-1-\nu)!(k+2c-\nu)!}{(j-\nu)!(k-1-\nu)!(2k+2c-\nu)!\nu!}.
\end{equation}
Next, with \eqref{5.6}, \eqref{5.8} and Lemma~\ref{lem:2.4} we have
\begin{align*}
\sum_{n=k+1}^\infty\frac{1}{(n;k)_{{\bf z},1/2}} 
&= \sum_{j=0}^{2c}A_j\sum_{n=1}^\infty\frac{1}{(n+2k+2c-j)\cdots(n+2c-j)}\\
&= \sum_{j=0}^{2c}A_j\cdot\frac{(2c-j)!}{2k(2k+2c-j)!}.
\end{align*}
Finally, this combined with \eqref{5.12} gives the desired evaluation 
\eqref{5.5}.
\end{proof}

\begin{proof}[Proof of Theorem~\ref{thm:5.1}]
We use Theorem~\ref{thm:2.3} with $\beta=2$ and $d=0$, so that $z_k=(k+c)^2$.
Furthermore, we set $x=0$ and $\alpha=\frac{1}{2}$, and we let $c\geq 0$ be 
an integer. From \eqref{5.5a} we have
\begin{align}
(k;k-1)_{{\bf z},1/2} &= (2k-1+2c)\cdots(k+1++2c)(k+c)(k-1)\cdots 2\cdot 1\label{5.13}\\
&= (k+c)\frac{(2k-1+2c)!(k-1)!}{(k+2c)!}.\nonumber
\end{align}
Based on the observation that the right-hand side of \eqref{5.5}, when 
multiplied by $2(k+c)!(2k+2c)!/(k-1)!$, seems to be a polynomial in $k$ with
integer coefficients, we denote
\begin{equation}\label{5.14}
\widetilde{P}_c(k):=\frac{2(k+c)!(2k+2c)!}{(k-1)!}
\sum_{n=k+1}^\infty\frac{1}{(n;k)_{{\bf z},1/2}}.
\end{equation}
Hence with \eqref{2.7} and \eqref{5.13}, \eqref{5.14} we get
\[
\gamma_k(0)=\frac{1}{(k+c)^2}\cdot\frac{(k+2c)!}{(k+c)(2k-1+2c)!(k-1)!}
+\frac{(k-1)!\widetilde{P}_c(k)}{2(k+c)!(2k+2c)!}.
\]
Rewriting this, we find
\begin{equation}\label{5.15}
\gamma_k(0)=\frac{(k-1)!}{4(k+c)^2(k+c-1)!(2k+2c-1)!}P_c(k),
\end{equation}
where 
\begin{equation}\label{5.16}
P_c(k) = \frac{4(k+2c)!(k+c-1)!}{(k+c)(k-1)!^2} + \widetilde{P}_c(k). 
\end{equation}
This last identity, with \eqref{5.14} and \eqref{5.5}, can be seen to be the
same as \eqref{5.2}.

Next we note that
\[
\prod_{\ell=1}^{k-1}\left(-z_\ell^2\right) 
= (-1)^{k-1}\prod_{\ell=1}^{k-1}(\ell+c)^2 = (-1)^{k-1}\frac{(k+c-1)!^2}{c!^2},
\]
and so, with \eqref{5.15} and \eqref{2.15} we get 
\begin{align}
\zeta_{\bf z}(\tfrac{3}{2})
&=\sum_{k=1}^\infty(-1)^{k-1}\frac{(k-1)!(k+c-1)!}{4(k+c)^2c!^2(2k+2c-1)!}P_c(k)\label{5.17}\\
&=\frac{1}{2c!^2}\sum_{k=1}^\infty(-1)^{k-1}\frac{(k-1)!(k+c)!}{(k+c)^2(2k+2c)!}P_c(k).\nonumber
\end{align}
Finally, by \eqref{2.4} and \eqref{5.0} we have 
$\zeta_{\bf z}(\tfrac{3}{2})=\zeta_c(3)$, and upon rewriting the right-most 
term in \eqref{5.17}, we get the desired identity \eqref{5.1}.
\end{proof}

\section{Identities for even powers of $\pi$}

In this section we use Theorem~\ref{thm:2.3} again, but this time with
$z_n=(n+\frac{1}{2})^2$ and $\alpha=0$. While in Section~5 we considered only
the case $x=0$, we will now make full use of the identity \eqref{2.15}. Before
we can state the main result of this section, we need the following definition.
For integers $K\geq 1$ and $\nu\geq 1$ we denote
\begin{equation}\label{6.1}
\widetilde{H}_K(\nu)
:=\sum_{K\geq k_1>\cdots>k_{\nu}\geq 1}\frac{1}{(2k_1+1)^2\cdots(2k_{\nu}+1)^2},
\end{equation}
and we set $\widetilde{H}_K(0)=1$ for all integers $K\geq 1$. These numbers can
be seen as a type of generalized harmonic numbers.

\begin{theorem}\label{thm:6.1}
For any integer $\mu\geq 0$ we have
\begin{align}
&\big(1-4^{-\mu-1}\big)\zeta(2\mu+2) 
=1+\sum_{k=1}^\infty\frac{(-1)^{k+\mu-1}}{16^k(2k+1)^2}\binom{2k}{k}\label{6.2}\\
&\times\bigg(\frac{10k^3+9k^2-k+1}{2k-1}\widetilde{H}_{k-1}(\mu)
+4k(k+1)\sum_{j=1}^{\mu}\frac{(-1)^j}{(2k+1)^{2j}}\widetilde{H}_{k-1}(\mu-j)\bigg).\nonumber
\end{align}
\end{theorem}

Before proving this result, we state the two smallest cases separately. To do 
so, we use the fact that $\zeta(2)=\pi^2/6$ and $\zeta(4)=\pi^4/90$.

\begin{corollary}\label{cor:6.2}
\begin{align}
\frac{\pi^2}{8} &= 1+\sum_{k=1}^\infty\frac{(-1)^{k-1}}{16^k}\binom{2k}{k}
\frac{10k^3+9k^2-k+1}{(2k-1)(2k+1)^2},\label{6.3}\\
\frac{\pi^4}{96} &= 1+\sum_{k=1}^\infty\frac{(-1)^{k-1}}{16^k}\binom{2k}{k}
\frac{1}{(2k+1)^2}\label{6.4}\\
&\qquad\times\bigg(\frac{4k(k+1)}{(2k+1)^2}-\frac{10k^3+9k^2-k+1}{2k-1}
\sum_{j=1}^{k-1}\frac{1}{(2j+1)^2}\bigg).\nonumber
\end{align}
\end{corollary}

More generally, using Euler's formula
\[
\zeta(2n) = \frac{(-1)^{n-1}2^{2n-1}B_{2n}}{(2n)!}\cdot\pi^{2n},
\]
where $B_{2n}$ is the $2n$th Bernoulli number, we can write the left-hand side
of \eqref{6.2} as a rational multiple of $\pi^{2\mu+2}$.

For the proof of Theorem~\ref{thm:6.1} we require the following series
evaluation.

\begin{lemma}\label{lem:6.3}
For any integer $k\geq 1$ we have
\begin{equation}\label{6.5}
\sum_{n=k+1}^\infty\frac{n(n+1)(n-k+1)!}{(n+k+1)!}
=\frac{2k^3+5k^2+3k+1}{(2k-1)(2k+1)(2k+1)!}.
\end{equation}
\end{lemma}

\begin{proof}
We denote the series in \eqref{6.5} by $S$ and shift the summation by $k+1$
units, obtaining
\begin{align}
S &= \sum_{n=0}^\infty(n+k+1)(n+k+2)\frac{n!}{(n+2k+2)!}\label{6.6}\\
&= \sum_{n=0}^\infty\frac{n^2\cdot n!}{(n+2k+2)!}
+(2k+3)\sum_{n=0}^\infty\frac{n\cdot n!}{(n+2k+2)!} \nonumber\\
&\qquad\qquad+(k+1)(k+2)\sum_{n=0}^\infty\frac{n!}{(n+2k+2)!}\nonumber\\
&= S_2 + (2k+3)S_1 + (k+1)(k+2)S_0.\nonumber
\end{align}
To evaluate the series $S_0, S_1, S_2$, we use the Gaussian hypergeometric
series,
\begin{equation}\label{6.7}
_2F_1(a,b;c;z) = \sum_{n=0}^\infty\frac{(a)_n(b)_n}{(c)_n}\cdot\frac{z^n}{n!},
\end{equation}
where $(x)_n=x(x+1)\cdots(x+n-1)$ for $n\geq 1$ and $(x)_0=1$, and apply the
well-known identity
\begin{equation}\label{6.8}
_2F_1(a,b;c;1) = \frac{\Gamma(c)\Gamma(c-a-b)}{\Gamma(c-a)\Gamma(c-b)}\qquad
({\rm Re}(c-a-b)>0),
\end{equation}
known as Gauss's theorem;
see, e.g., \cite[Eq.~15.4.20]{DLMF}. Applying \eqref{6.7} and then \eqref{6.8},
we first see that
\begin{equation}\label{6.9}
S_0=\sum_{n=0}^\infty\frac{n!}{(n+2k+2)!}
=\frac{1}{(2k+2)!}\,_2F_1(1,1;2k+3;1) =\frac{1}{(2k+1)(2k+1)!}.
\end{equation}
Next we note that 
\[
\sum_{n=0}^\infty\frac{n\cdot n!}{(n+2k+2)!}
= \sum_{n=0}^\infty\frac{(n+1)!}{(n+2k+2)!} - S_0,
\]
and with \eqref{6.7} and \eqref{6.8} we get
\[
\sum_{n=0}^\infty\frac{(n+1)!}{(n+2k+2)!}
=\frac{1}{(2k+2)!}\,_2F_1(1,2;2k+3;1) =\frac{1}{2k(2k+1)!},
\]
so that
\begin{equation}\label{6.10}
S_1 = \frac{1}{2k(2k+1)!} - \frac{1}{(2k+1)(2k+1)!} = \frac{1}{2k(2k+1)(2k+1)!}.
\end{equation}
Next we have
\begin{align}
S_2 &= \sum_{n=0}^\infty\frac{n^2\cdot n!}{(n+2k+2)!}
= \sum_{n=0}^\infty\frac{\big(1-3(n+1)+(n+1)(n+2)\big)\cdot n!}{(n+2k+2)!}\label{6.11}\\
&= S_0-3(S_0+S_1)+\sum_{n=0}^\infty\frac{(n+2)!}{(n+2k+2)!}.\nonumber
\end{align}
Using \eqref{6.7} and \eqref{6.8} again, we get
\[
\sum_{n=0}^\infty\frac{(n+2)!}{(n+2k+2)!}
=\frac{2}{(2k+2)!}\,_2F_1(1,3;2k+3;1) =\frac{2}{(2k-1)(2k+1)!}.
\]
Finally, combining this with \eqref{6.11} and then with \eqref{6.6}, \eqref{6.9}
and \eqref{6.10}, we see after some routine manipulations that $S$ equals the
right-hand side of \eqref{6.5}.
\end{proof}

\begin{proof}[Proof of Theorem~\ref{thm:6.1}]
We begin with using the definition \eqref{2.3}, obtaining
\begin{align}
(n;k)_{{\bf z},0} 
&= \prod_{\ell=1}^k\left((n+\tfrac{1}{2})^2-(\ell+\tfrac{1}{2})^2\right)
=\frac{1}{4^k}\prod_{\ell=1}^k\left((2n+1)^2-(2\ell+1)^2\right)\label{6.12}\\
&= \frac{1}{4^k}(2(n+k)+2)(2(n+k))(2(n+k)-2)\cdots(2n+4)\nonumber\\
&\qquad\qquad\times(2n-2)(2n-4)\cdots(2n-2k)\nonumber\\
&= \frac{1}{4^k}\cdot 2^{2k}\cdot\frac{(n+k+1)!}{(n+1)!}\cdot\frac{(n-1)!}{(n-k-1)!}
= \frac{(n+k+1)!}{n(n+1)(n-k-1)!},\nonumber
\end{align}
and by Lemma~\ref{lem:6.3} we have
\begin{equation}\label{6.13}
\sum_{n=k+1}^\infty\frac{1}{(n;k)_{{\bf z},0}}
=\frac{2k^3+5k^2+3k+1}{(2k-1)(2k+1)(2k+1)!}.
\end{equation}
By \eqref{6.12} we also have
\begin{equation}\label{6.14}
(k;k-1)_{{\bf z},0} = \frac{(k+k)!}{k(k+1)\cdot 0!} = \frac{2(2k-1)!}{k+1}.
\end{equation}
These identities, combined with \eqref{2.7}, give
\[
\gamma_k(x)=\frac{k+1}{2(2k-1)!}\cdot\frac{1}{(k+\tfrac{1}{2})^2-x}
+\frac{2k^3+5k^2+3k+1}{(2k-1)(2k+1)(2k+1)!},
\]
and since
\[
\frac{1}{(k+\tfrac{1}{2})^2-x} 
= \frac{4}{(2k+1)^2}\cdot\frac{1}{1-\frac{4x}{(2k+1)^2}},
\]
we get
\begin{equation}\label{6.15}
\gamma_k(x)=\frac{10k^3+9k^2-k+1}{(2k-1)(2k+1)(2k+1)!}
+\frac{4k(k+1)}{(2k+1)(2k+1)!}
\sum_{j=1}^\infty\left(\frac{2}{2k+1}\right)^{2j}x^j.
\end{equation}
Next we have
\begin{align*}
\prod_{\ell=1}^{k-1}\left(x-z_l\right)
&=\frac{1}{4^{k-1}}\prod_{\ell=1}^{k-1}\left(4x-(2\ell+1)^2\right) \\
&=(-1)^{k-1}\frac{(2k)!^2}{4^{2k-1}k!^2}
\prod_{\ell=1}^{k-1}\left(1-\frac{4x}{(2\ell+1)^2}\right),
\end{align*}
and thus, by \eqref{6.1}, we get
\begin{equation}\label{6.16}
\prod_{\ell=1}^{k-1}\left(x-z_l\right) = (-1)^{k-1}\frac{(2k)!^2}{4^{2k-1}k!^2}
\sum_{\nu=0}^{k-1}(-4)^\nu\widetilde{H}_{k-1}(\nu)x^{\nu}.
\end{equation}
Now we proceed in analogy to the proof of Theorem~\ref{4.1} and let $R(x)$ be
the right-hand side of \eqref{2.15}. Multiplying \eqref{6.15} and \eqref{6.16}
and summing over $k$, we get
\begin{align}
R(x)&=\sum_{k=1}^\infty(-1)^{k-1}\frac{(2k)!^2}{4^{2k-1}k!^2}
\cdot\frac{1}{(2k+1)(2k+1)!}\cdot S_k(x)\label{6.17}\\
&= \sum_{k=1}^\infty\frac{(-1)^{k-1}}{16^k}\binom{2k}{k}\frac{4}{(2k+1)^2}\cdot S_k(x)\nonumber,
\end{align}
where
\begin{align}
S_k(x)&=\frac{10k^3+9k^2-k+1}{2k-1}\sum_{\nu=0}^{k-1}(-4)^\nu\widetilde{H}_{k-1}(\nu)x^\nu\label{6.18}\\
&\qquad+4k(k-1)\sum_{\mu=1}^\infty\bigg(\sum_{j=1}^{\mu}\left(\frac{2}{2k+1}\right)^{2j}(-4)^{\mu-j}\widetilde{H}_{k-1}(\mu-j)\bigg)x^\mu.\nonumber
\end{align}
Equating coefficients of $x^\mu$ in \eqref{6.17}, \eqref{6.18} with the 
left-hand side of \eqref{2.15}, we get
\begin{align}
\zeta_{\bf z}(\mu+1)
&=\sum_{k=1}^\infty\frac{(-1)^{k-1}\cdot 4}{16^k(2k+1)^2}\binom{2k}{k}
\bigg(\frac{10k^3+9k^2-k+1}{2k-1}(-4)^\mu\widetilde{H}_{k-1}(\mu)\label{6.19}\\
&\qquad+4^{\mu+1}k(k+1)\sum_{j=1}^{\mu}\frac{(-1)^{\mu-j}}{(2k+1)^{2j}}\widetilde{H}_{k-1}(\mu-j)\bigg).\nonumber
\end{align}
Finally, by the definition \eqref{2.4} we have
\begin{align*}
\zeta_{\bf z}(\mu+1)
&=\sum_{n=1}^\infty\frac{1}{(n+\tfrac{1}{2})^{2\mu+2}}
=4^{\mu+1}\sum_{n=1}^\infty\frac{1}{(2n+1)^{2\mu+2}}\\
&=4^{\mu+1}\bigg(\sum_{n=1}^\infty\frac{1}{n^{2\mu+2}}
-\sum_{n=1}^\infty\frac{1}{(2n)^{2\mu+2}} - 1\bigg) \\
&= \big(4^{\mu+1}-1)\zeta(2\mu+2) - 4^{\mu+1}.
\end{align*}
Combining this with \eqref{6.19} and dividing both sides by $4^{\mu+1}$, we
finally get \eqref{6.2}, as desired.
\end{proof}

In closing, we note that an identity very similar to \eqref{6.2} was earlier
obtained
by Leshchiner \cite[Eq.~(4b)]{Le}. One basic difference lies in the fact that
the analogue of the multiple generalized harmonic sum \eqref{6.1} used by
Leshchiner has $k_\nu\geq 0$ in the summation (using our notation, which differs
from Leshchiner's). In analogy to the identities \eqref{6.3} and \eqref{6.4}
above, the two smallest cases of Eq.~(4b) in \cite{Le} are
\begin{align*}
\frac{\pi^2}{10} 
&= 1+\sum_{k=1}^\infty\frac{(-1)^k}{16^k(2k+1)^2}\binom{2k}{k},\\
\frac{\pi^4}{96} 
&= 1+\sum_{k=1}^\infty\frac{(-1)^k}{16^k(2k+1)^2}\binom{2k}{k}
\bigg(\frac{1}{(2k+1)^2}-\frac{5}{4}\sum_{j=0}^{k-1}\frac{1}{(2j+1)^2}\bigg).
\end{align*}
In spite of the similarities, Theorem~\ref{thm:6.1} above appears to be new.

\end{document}